\documentclass[12pt]{elsarticle}
\usepackage{amscd}
\usepackage{amssymb}
\usepackage{a4wide}
\usepackage{amstext}
\usepackage{amsthm}
\usepackage{mathrsfs}
\usepackage{amsmath}

\renewcommand{\phi}{\varphi}

\newtheorem{Thm}{Theorem}[section]
\newtheorem{thm}[Thm]{Theorem}
\newtheorem{lem}[Thm]{Lemma}

\newtheorem{rem}[Thm]{Remark}

\begin{document}  
\sloppy                             
\title{Weissler and Bernoulli type inequalities in Bergman spaces}

\author[1]{Anton D. Baranov}
\ead{anton.d.baranov@gmail.com}
\author[1,2]{Ilgiz R. Kayumov}
\ead{ikayumov@gmail.com}
\author[2,3]{Diana M. Khammatova\corref{cor1}}
\ead{dianalynx@rambler.ru}
\cortext[cor1]{Corresponding author}
\author[1]{Ramis Sh. Khasyanov}
\ead{hasbendshurich@gmail.com}

\affiliation[1]{organization={St.~Petersburg State University},
	addressline={7-9 Universitetskaya Embankment},
	postcode={199034},
	city={Saint Petersburg},
	country={Russia}}
\affiliation[2]{organization={Kazan Federal University},
	addressline={18 Kremlyovskaya street},
	postcode={420008},
	city={Kazan},	
	country={Russia}}
\affiliation[3]{organization={Moscow Polytechnic University},
	addressline={38 Bolshaya Semyonovskaya street},
	postcode={107023},
	city={Moscow},	
	country={Russia}}



\begin{abstract}
We consider Weissler type inequalities for Bergman spaces with general radial weights
and give conditions on the weight $w$ in terms of its moments
ensuring that $\|f_r\|_{A^{2n}(w)}\leq \|f\|_{A^2(w)}$ whenever $n\in \mathbb{N}$ and $0< r\le 1/\sqrt{n}$.
For noninteger exponents a special case of this inequality is proved 
which can be considered as a certain analog of the Bernoulli inequality.
An example of a monotonic weight is constructed for which these inequalities are no longer true. 
\end{abstract}

\begin{keyword}
	Bergman space \sep Weissler inequality \sep Bernoulli inequality
\end{keyword}

\maketitle

\section{Introduction}
We consider Weissler type inequalities for Bergman spaces 
in the disc with general radial weights. While such questions were studied extensively
for the classical weights, the case of general weights apparently was not previously addressed.
The aim of this note is to propose regularity/convexity conditions on the weight,
expresses in terms of its moments, ensuring that some special cases of the Weissler type inequality hold true.
We show that certain regularity is required since there exist monotonic weights for which
even these special cases of the Weissler type  inequality are false.

To describe our results, let us recall some definitions and classical inequalities. 
A function $f$, analytic in the unit disk $\mathbb{D}$, belongs to the Hardy space  $H^p$, $0<p <\infty$, if
$$
\|f\|_{H^p}:=\sup\limits_{0<r<1}\left(\frac1{2\pi}\int\limits_0^{2\pi}\lvert f(re^{i\theta})\rvert^p\,d\theta\right)^{\frac1p}<\infty.
$$ 		

For $r\in (0,1)$ let $f_r(z) = f(rz)$. 
A well-known result of F.\,B.~Weissler \cite{wei} states that for the Hardy spaces $H^p$ and $H^q$ ($0< p < q$)
$$
\|f_r \|_{H^q}\leq \|f \|_{H^p} \  \ \text{for any}\ f\in H^p 
\qquad \Longleftrightarrow \qquad r\leq \sqrt{\frac pq}\leq 1.
$$

Given a summable nonnegative function $w$ on $[0,1)$, we say that a function $f$, analytic in the unit disk $\mathbb{D}$, 
belongs to the weighted Bergman space $A^p(w)$ if
$$
\|f\|_{A^p(w)} :=\left( \frac1{2\pi}\int\limits_{\mathbb{D}}  \lvert f(z)\rvert^p  w(\lvert z\rvert) \, dA(z)\right)^{1/p} <\infty,
$$
where $A$ denotes the planar Lebesgue measure. We define the moments of the weight $w$ as
$$
h_m = \int\limits_0^1\rho^{m+1}w(\rho)\,d\rho, \qquad m\geq 0.
$$
In what follows we always assume the normalization $h_0=1$.

The weights $w_{\alpha}(r) = 2(\alpha-1)(1-r^2)^{\alpha-2}$, where $\alpha>1$, 
are called classical weights. We will denote Bergman spaces with the weight $w_{\alpha}$ as $A_{\alpha}^p$. 
For the theory of Bergman spaces see, e.g., \cite{hed}.

The classical Carleman inequality
$$
\left(\sum\limits_{n=0}^{\infty}\frac{\lvert a_n\rvert^2}{n+1}\right)^{\frac12} \leq \|f\|_{H^1}
$$
for $f(z) = \sum\limits_{n=0}^{\infty} a_nz^n \in H^1$ 
implies that $\|f\|_{A_2^2}\leq \|f\|_{H^1}$ whence it is easy to deduce that
$$
 \|f\|_{A_2^{2p}}\leq \|f\|_{H^p},\qquad 0<p<\infty.
$$
The following generalization of the Carleman inequality was proved by J.~Burbea \cite{bur}. 
Let $k\in \mathbb{N}$, $k\ge 2$, and $p=\frac2k$. 
Then, for  $f(z) = \sum\limits_{n=0}^{\infty}a_nz^n\in H^p$,
\begin{equation}
\label{hypKul}
\|f\|_{A_{2/p}^2} = \left(\sum\limits_{n=0}^{\infty}\frac{\lvert a_n\rvert^2}{c_{2/p}(n)}\right)^{\frac12} \leq \|f\|_{H^p},
\qquad c_{\beta}(n) = C_{n+\beta-1}^n.
\end{equation}


It was conjectured in \cite{bre1} that inequality \eqref{hypKul} holds for all $0<p\leq 2$.  	
An inequality similar to \eqref{hypKul}, but with a constant slightly worse than 1 on the right,
was proved in \cite{bre1, lli}. 

Contractive inequalities for Bergman spaces $A_{\alpha}^p$ 
were also studied by many authors. In particular, 
F.~Bayart, O.\,F.~Brevig, A.~Haimi, J.~Ortega-Cerd\`a and K.-M.~Perfekt \cite{bre2} proved the following
counterpart of the Weissler inequality. Let $0 < p\leq q <\infty$ and $\alpha = \frac{n+1}2$ 
for some $n\in\mathbb{N}$. Then for $f\in A_{\alpha}^p$
\begin{equation}
\label{wb}
	\|f_r\|_{A_{\alpha}^q}\leq \|f\|_{A_{\alpha}^p}
	\  \  \text{for any}\ f\in A_{\alpha}^p
	\quad \Longleftrightarrow \quad r\leq \sqrt{\frac pq}\leq 1.
\end{equation}

A remarkable progress in these problems was achieved in 2022 (after the first version of the present note was finished).
A.\,Kulikov \cite{kul} proved the conjectured inequality  \eqref{hypKul} for all $0<p\le 2$, he also showed
that it is sharp in every coefficient. Using the results and methods of the paper by Kulikov,
P. Melentijevi\'c \cite{mel} proved that \eqref{wb} is true for any $0 < p < q < \infty$ such that $q\ge 2$ and 
for any $\alpha>1$. Moreover, \eqref{wb} holds for $q<2$ as well if we assume that $f$ is zero-free, while in general
only a lower bound for $r$ was found in this case.

In the present note we are interested in analogs of \eqref{wb} for Bergman spaces 
with general radial weights. Our goal is to find conditions on the weight 
for which a Weissler type inequality is true. We give such conditions in terms of the moments  $h_{2m}$ of the weight $w$. 
While our present results deal with very special situations (even integer exponents or a specific choice of 
a function) we conjecture that under these or similar conditions on the weight
the results can be extended to the case of general exponents. 

Note that  by the Cauchy inequality, one always has $h^2_{2m} \le h_{2(m-1)} h_{2(m+1)}$. 
Our condition in the next theorem is a converse (in a sense) inequality.

\begin{thm}\label{even}
	Let $w$ be a Bergman weight satisfying
	\begin{equation}\label{weakcond}
		\frac{h_{2m}}{h_{2(m-1)}}\geq \frac m{m+1}\frac{h_{2(m+1)}}{h_{2m}}
	\end{equation}
	all $m\geq 1$.  Then for any $n\in\mathbb{N}$ we have 
	\begin{equation}\label{conc}
		\|f_r\|_{A^{2n}(w)}\leq \|f\|_{A^2(w)} 
		\  \ \text{for any}\ f\in A^2(w)
		\quad \Longleftrightarrow \quad r\leq \frac{1}{\sqrt{n}} \leq 1.
	\end{equation}
\end{thm}

It is easy to see that \eqref{weakcond} holds for all classical weights $w_\alpha$.

We now turn to Weissler type inequalities between $A^2(w)$ and $A^{2q}(w)$ where $q>1$ 
is an arbitrary real number. This problem is much more complex, because now one cannot use combinatorics. 
In fact, a natural conjecture arises that the inequality $\|f_r\|_{A^{2q}(w)}\leq \|f\|_{A^2(w)}$ holds for 
$r=\frac{1}{\sqrt{q}}$. In other words,
\begin{equation}
	\label{ram}
	\frac1{2\pi}\int\limits_0^1\int\limits_0^{2\pi}\left\lvert f^{q}\left(\frac{\rho e^{i\theta}}{\sqrt{q}}\right)\right\rvert^2
	\rho w(\rho)\, d\theta\,d\rho \leq \left(\frac1{2\pi}\int\limits_0^1\int\limits_0^{2\pi}\lvert f(\rho e^{i\theta})\rvert^2\rho w(\rho)\, d\theta\,d\rho\right)^{q}.
\end{equation}

To have the possibility of taking the powers we consider functions nonvanishing in $\mathbb{D}$. 
They can be represented as $f(z) = e^{\phi(z)}$. Let $\phi(z) =\sum_{n=0}^{\infty}a_k z^k$. 
Since the norms in $A_{\alpha}^n$ are expressed in terms of modulus of coefficients, 
we can assume that all $a_k\geq0$. For such functions
$$
	f^{q}\left(z\right) = e^{q \phi(z)} =  \sum_{n=0}^{\infty}\frac{q^n}{n!}\left(\sum_{k=0}^{\infty}a_k z^k\right)^n 
	= \sum_{n=0}^{\infty}\frac{q^n}{n!}\sum_{k=0}^{\infty}\sum_{j_1+\ldots+j_n=k}a_{j_1}\cdot\ldots\cdot a_{j_n} z^k.
$$
Changing the order of summation and substituting the value of the argument, we have
$$
f^{q}\left(\frac{\rho e^{i\theta}}{\sqrt{q}}\right) = \sum_{n=0}^{\infty}\left(\sum_{k=0}^{\infty} 
\frac{q^k}{k!}\sum_{j_1+\ldots+j_k=n}a_{j_1}\cdot\ldots\cdot a_{j_k}\right) \frac{ \rho^ne^{i\theta n}}{q^{n/2}}.
$$

We introduce the following notation.
$$
g_{nk} = \frac1{k!}\sum_{j_1+\ldots+j_k=n}a_{j_1}\cdot\ldots\cdot a_{j_k}.
$$
Then the left-hand side of \eqref{ram} can be represented as
$$
\frac1{2\pi}\int\limits_0^1\int\limits_0^{2\pi}\left\lvert f^{q}\left(\frac{\rho e^{i\theta}}{\sqrt{q}}\right)\right\rvert^2\rho w(\rho)\, d\theta\,d\rho =
\sum_{n=0}^{\infty}\frac1{q^n}\left(\sum_{k=0}^{\infty} q^k g_{nk}\right)^2 h_{2n}.
$$
Similarly, for the right-hand side of \eqref{ram} we have
$$
\left(\frac1{2\pi}\int\limits_0^1\int\limits_0^{2\pi}\lvert f(\rho e^{i\theta})\rvert^2\rho w(\rho)\, d\theta\,d\rho\right)^{q} = \left(\sum_{n=0}^{\infty}\left(\sum_{k=0}^{\infty} g_{nk}\right)^2 h_{2n}\right)^q.
$$
Therefore, the desired inequality takes the following form:
$$
\sum_{n=0}^{\infty}\frac1{q^n}\left(\sum_{k=0}^{\infty} q^k g_{nk}\right)^2 h_{2n} \leq \left(\sum_{n=0}^{\infty}\left(\sum_{k=0}^{\infty} g_{nk}\right)^2 h_{2n}\right)^q.
$$

This inequality looks, in general, quite complicated,
so it seems logical to consider some particular cases. For example, one can take the function $f(z) = e^z$. In this case $g_{nk} = \frac{\delta_{nk}}{k!}$ and the desired inequality is
\begin{equation}
	\label{bern}
	\sum\limits_{n=0}^{\infty}\frac{q^n}{(n!)^2}h_{2n}\leq \left(\sum\limits_{n=0}^{\infty}\frac{h_{2n}}{(n!)^2}\right)^q.
\end{equation}
Although this is only a particular case, this inequality seems to be of independent interest  because it 
can be considered as an analog of the classical Bernoulli inequality for the moment sequences. It should be 
emphasized that it is essential here that $h_{2n}$ is a sequence of moments of some weight and also
that this inequality need not be true for an arbitrary (even monotonic) weight.

The following theorem gives a sufficient condition for the inequality \eqref{bern}. We have to replace  condition
\eqref{weakcond} from Theorem~\ref{even} by a stronger condition on the moments.

\begin{thm}\label{mainth}
	Let $w $ be a Bergman weight. If the inequality
	\begin{equation}\label{cond}
		\frac{h_{2m}}{h_{2(m-1)}}\geq \frac{h_{2(m+1)}}{(m+1)h_{2(m-1)}}+\frac m{m+1}\frac{h_{2(m+1)}}{h_{2m}}
	\end{equation}
	holds for all $m\geq 1$, then
	\begin{equation}\label{concl}
		\sum\limits_{n=0}^{\infty}\frac{q^n}{(n!)^2}h_{2n}\leq \left(\sum\limits_{n=0}^{\infty}\frac{h_{2n}}{(n!)^2}\right)^q,\qquad q\geq1.
	\end{equation}
\end{thm}
	
	In particular, the estimate \eqref{cond} holds for all standard Bergman weights $w_{\alpha}$;
	for them it turns into equality. Thus, it looks plausible that the weights satisfying \eqref{cond} 
	are a correct class for generalizations of \eqref{wb} to general weights. 

However, \eqref{conc} and \eqref{concl} do not  hold for arbitrary weights and we indeed 
need to impose some regularity conditions. 

\begin{thm}\label{counter}
	There exists a monotonically decreasing weight $w$ such that the inequality 
	$$
	\|f_r\|_{A^{2q}(w)}\leq \|f\|_{A^2(w)}
	$$
	does not hold for $f(z) = e^z$, $r=\frac{1}{\sqrt{q}}$ and $q=2$ as well as for $q\in (1, 1+\varepsilon)$ 
	for some $\varepsilon >0$.  
\end{thm}

The paper is organized as follows. In Section 2 we prove Theorem~\ref{even} which gives the counterpart 
of Weissler inequality for Bergman spaces when $q$ is even and $p=2$. In Section 3 two auxiliary lemmas are proved, while
Section 4 is devoted to the proof of  Theorems~\ref{mainth} and \ref{counter}.
\bigskip


\section{Proof of Theorem~\ref{even}}

We need to prove the following inequality.
\begin{equation}\label{bb}
	\frac1{2\pi}\int\limits_0^1\int\limits_0^{2\pi}\left\lvert f^n\left(\frac{\rho e^{i\theta}}{\sqrt{n}}\right)\right\rvert^2\rho w(\rho)\, d\theta\,d\rho \leq \left(\frac1{2\pi}\int\limits_0^1\int\limits_0^{2\pi}\lvert f(\rho e^{i\theta})\rvert^2\rho w(\rho)\, d\theta\,d\rho\right)^n.
\end{equation}

Let $f(z) = \sum\limits_{k=0}^{\infty}a_kz^k$. Then $f^n(z)$ can be represented as
$$
f^n(z) = \sum\limits_{k=0}^{\infty}\left(\sum\limits_{j_1+\ldots+j_n=k}a_{j_1}\cdot\ldots\cdot a_{j_n}\right)z^k.
$$
The integrals in \eqref{bb} can be calculated using Parseval identity. Thus, the left-hand side of \eqref{bb} is
$$
	\int\limits_0^1\rho w(\rho)\sum\limits_{k=0}^{\infty} \left(\sum\limits_{j_1+\ldots+j_n=k}a_{j_1}\cdot\ldots\cdot a_{j_n}\right)^2\frac{\rho^{2k}}{n^k}\,d\rho = \sum\limits_{k=0}^{\infty}\frac1{n^k}\left(\sum\limits_{j_1+\ldots+j_n=k}a_{j_1}\cdot\ldots\cdot a_{j_n}\right)^2h_{2k}
$$
and the right-hand side of \eqref{bb} is
\begin{multline*}
	\left(\int\limits_0^1\rho w(\rho)\sum\limits_{k=0}^{\infty}a_k^2 \rho^{2k}\,d\rho\right)^n =  \left(\sum\limits_{k=0}^{\infty}a_k^2\int\limits_0^1\rho^{2k+1} w(\rho)\,d\rho\right)^n = \\ = \sum\limits_{k=0}^{\infty} \sum\limits_{j_1+\ldots+j_n=k}a_{j_1}^2\cdot\ldots\cdot a_{j_n}^2 h_{2j_1}\cdot\ldots\cdot h_{2j_n}.
\end{multline*}	
From the Cauchy inequality
\begin{multline*}
	\sum\limits_{k=0}^{\infty}\frac{h_{2k}}{n^k}\left(\sum\limits_{j_1+\ldots+j_n=k}a_{j_1}\cdot\ldots\cdot a_{j_n}
	\right)^2\leq \\ 
	\leq \sum\limits_{k=0}^{\infty}\frac{h_{2k}}{n^k}\sum\limits_{j_1+\ldots+j_n=k}a_{j_1}^2\cdot\ldots\cdot a_{j_n}^2 h_{2j_1}\cdot\ldots\cdot h_{2j_n} \cdot \sum\limits_{j_1+\ldots+j_n=k} \frac1{h_{2j_1}\cdot\ldots\cdot h_{2j_n}}.
\end{multline*}
Now it is sufficient to prove that
$$
\sum\limits_{j_1+\ldots+j_n=k} \frac1{h_{2j_1}\cdot\ldots\cdot h_{2j_n}} \leq \frac{n^k}{h_{2k}},
$$
and then the desired inequality will hold term by term.

Inequality~\eqref{weakcond} implies that for $m\ge n$ 
$$
h_{2n} h_{2m} \geq \frac n{m+1}h_{2(n-1)}h_{2(m+1)}.
$$

We need to estimate the product $h_{2j_1}\cdot\ldots\cdot h_{2j_n}$, where $j_1+\ldots+j_n=k$. 
Without loss of generality we assume that $j_1\leq j_2$. Then
\begin{multline*}
	h_{2j_1} h_{2j_2}\geq \frac{j_1}{j_2+1}h_{2(j_1-1)}h_{2(j_2+1)}\geq \frac{j_1(j_1-1)}{(j_2+1)(j_2+2)}h_{2(j_1-2)}h_{2(j_2+2)}\geq \\ \geq \frac{j_1!}{(j_2+1)\ldots(j_2+j_1)}h_0h_{2(j_2+j_1)} = \frac{j_1!j_2!}{(j_1+j_2)!}h_0h_{2(j_2+j_1)}
\end{multline*}
Next, assume, without loss of generality, that $j_1+j_2\geq j_3$. Then
$$
h_{2j_1} h_{2j_2}h_{2j_3}\geq \frac{j_1!j_2!}{(j_1+j_2)!}h_{2(j_1+j_2)}h_{2j_3} \geq \frac{j_1!j_2!j_3!}{(j_1+j_2+j_3)!}h_{2(j_1+j_2+j_3)}.
$$

%

Proceeding in this way we finally get
$$
h_{2j_1}\cdot\ldots\cdot h_{2j_n}\geq \frac{j_1!j_2!\ldots j_n!}{(j_1+j_2+\ldots+j_n)!}h_{2(j_1+j_2+\ldots+j_n)}=
\frac{j_1!j_2!\ldots j_n!}{k!}h_{2k}
$$
Therefore,
$$
\sum\limits_{j_1+\ldots+j_n=k}\frac{1}{h_{2j_1}\cdot\ldots\cdot h_{2j_n}} \leq \frac{k!}{h_{2k}}\sum\limits_{j_1+\ldots+j_n=k}\frac{1}{j_1!\cdot\ldots\cdot j_n!}= \frac{k!}{h_{2k}}\cdot\frac{n^k}{k!} = \frac{n^k}{h_{2k}}.
$$

The equality $\sum\limits_{j_1+\ldots+j_n=k}\frac{1}{j_1!\cdot\ldots\cdot j_n!} = \frac{n^k}{k!}$ is obtained by comparing the coefficients of
$$
e^{nz} = \sum\limits_{k=0}^{\infty}\frac{n^kz^k}{k!} = \sum\limits_{k=0}^{\infty}\sum\limits_{j_1+\ldots+j_m=k} \frac{z^k}{j_1!\cdot\ldots\cdot j_n!}.
$$
This finishes the proof of \eqref{bb}.

It remains to show the sharpness of the value $r=\frac{1}{\sqrt{n}}$. 
For this we use the standard test function $f(z) = 1+\varepsilon z$ where $\varepsilon$
can be taken arbitrarily small.  
Then the left-hand side of the desired inequality (i.e., $\|f_r\|^{2n}_{A^{2n}(w)}$) 
will be equal to
$$
\frac1{2\pi}\int\limits_0^1\int\limits_0^{2\pi}\big\lvert 1+\varepsilon r\rho e^{i\theta}\big\rvert^{2n}\rho w(\rho)\, d\theta\,d\rho = 1+n^2\varepsilon^2r^2\int\limits_0^1\rho^3w(\rho)\,d\rho+O(\varepsilon^4)
$$
and the right-hand side (i.e., $\|f\|^{2n}_{A^2(w)}$) will be given by
$$
\left(\frac1{2\pi}\int\limits_0^1\int\limits_0^{2\pi}\big\lvert1+\varepsilon \rho e^{i\theta}\big\rvert^{2}\rho w(\rho)\, d\theta\,d\rho\right)^n = 1+n\varepsilon^2 \int\limits_0^1\rho^3w(\rho)\,d\rho+O(\varepsilon^4).
$$
Comparing these values with $\varepsilon\to 0$, we see that if the desired inequality holds then $nr^2\leq 1$, so $r\leq \frac1{\sqrt{n}}$.

Let us show that the theorem applies to all classical weights $w_{\alpha}$. In this case
$$
h_{2n} = \frac{\Gamma(\alpha)\Gamma(n+1)}{\Gamma(\alpha+n)},
$$
whence
$$
\frac{h_{2(n+1)}}{(n+1)h_{2(n-1)}}+\frac n{n+1}\frac{h_{2(n+1)}}{h_{2n}} =\frac{n}{n+\alpha-1}= \frac{h_{2n}}{h_{2(n-1)}}.
$$
Thus, inequality \eqref{cond} turn into equality and, in particular, the weaker condition 
\eqref{weakcond} is satisfied. Theorem~\ref{even} is proved.
\bigskip


\section{Two lemmas}\label{sec3}

\begin{lem}\label{lem1}
	Let $g_n$ be a sequence of real numbers, such that $0< g_n\leq g_{n-1}$ for $n\geq 1$, $g_0=1$ and
	\begin{equation}\label{simpCond}
		\frac{g_{n}}{g_{n-1}}\geq \frac n{n+1}\frac{g_{n+1}}{g_{n}},\qquad n\geq 1.
	\end{equation}
	Then the inequality
	$$
	\sum\limits_{k=0}^{n}\frac{g_{n-k}\cdot g_{k+2}}{((n-k)!)^2(k!)^2(k+1)(k+2)}\leq \sum\limits_{k=0}^{n}\frac{g_{n-k+1}\cdot g_{k+1}}{((n-k)!)^2(k!)^2(n-k+1)(k+1)}.
	$$
	holds for all $n\geq 0$.

\end{lem}
\begin{proof}
	We can write
	\begin{multline*}
		T_n:=\sum\limits_{k=0}^{n}\frac{g_{n-k}\cdot g_{k+2}}{((n-k)!)^2(k!)^2(k+1)(k+2)}-\\-\sum\limits_{k=0}^{n}\frac{g_{n-k+1}\cdot g_{k+1}}{((n-k)!)^2(k!)^2(n-k+1)(k+1)} = \\
		=\frac{g_0\,g_{n+2}}{(n!)^2(n+1)(n+2)}-\frac{g_1\,g_{n+1}}{(n!)^2(n+1)}+
		\sum\limits_{k=0}^{n-1}\frac{g_{n-k}\,g_{k+2}(2k-n+1)}{((n-k)!)^2((k+1)!)^2(k+2)}.		
	\end{multline*}
	\noindent
	{\bf Case 1.} Let $n$ be even, i.e. $n=2m$. Then it can be represented as
	\begin{multline*}
		T_{2m}:=\frac{g_0\,g_{2m+2}}{((2m)!)^2(2m+1)(2m+2)}-\frac{g_{m+1}^2}{\left(m!\right)^2\left(\left(m+1\right)!\right)^2\left(m+1\right)}+\\+
		\sum\limits_{k=0}^{m-1}\frac{g_{m+k+2}\,g_{m-k}}{\left(\left(m+k+2\right)!\right)^2\left(\left(m-k\right)!\right)^2}(-2m+4k^2+8k+2).	
	\end{multline*}
	
	We denote
	$$
	t_k = \frac{g_{m+k+2}\,g_{m-k}}{\left(\left(m+k+2\right)!\right)^2\left(\left(m-k\right)!\right)^2}(-2m+4k^2+8k+2)
	$$
	and introduce the sequence $s_k$ the following way.
	$$
	s_0 = -\frac{g_{m+1}^2}{\left(m!\right)^2\left(\left(m+1\right)!\right)^2\left(m+1\right)},
	$$
	$$
	s_{k} = s_{k-1}\cdot\frac{\left(m-k+2\right)}{\left(m+k+1\right)}\frac{g_{m-k+1}\,g_{m+k+1}}{g_{m-k+2}\,g_{m+k}} +t_{k-1},\qquad 0\leq k\leq m.
	$$
	It is easy to show by induction that
	$$
	s_k = -\frac{(2k+1)g_{m+k+1}\,g_{m-k+1}}{\left(\left(m-k\right)!\right)^2\left(\left(m+k+1\right)!\right)^2\left(m-k+1\right)}.
	$$
	
	Therefore, $s_k<0$. One can note that at each iteration, the previously obtained result is multiplied by some coefficient and added to $t_{k-1}$, starting with the term $s_0$, which participates in $T_{2m}$.
	
	From the condition \eqref{simpCond} one has
	\begin{equation}\label{eqns}
		\frac p{q+1}g_{p-1}\,g_{q+1} \leq	g_{p}\, g_{q}\leq \frac{q}{p+1}g_{p+1}g_{q-1},\qquad p<q.
	\end{equation}
	Hence,
	$$
	g_{m-k+1}\,g_{m+k+1} \leq \frac{\left(m+k+1\right)}{\left(m-k+2\right)}g_{m-k+2}\,g_{m-k}.
	$$
	We conclude that
	$$
	s_k \geq s_{k-1}+t_{k-1}\geq s_{k-2}+t_{k-2}+t_{k-1}\geq\ldots\geq s_0+\sum_{q=0}^{k-1}t_k.
	$$
	Therefore,
	$$
	T_{2m}\leq \frac{g_0\,g_{2m+2}}{((2m)!)^2(2m+1)(2m+2)}+s_{m}.
	$$
	From the property \eqref{eqns}
	$$
	s_{m} = -\frac{(2m+1)g_2\,g_{2m+1}}{((2m+1)!)^2}\leq -\frac {g_0\,g_{2m+2}}{((2m)!)^2(2m+1)(2m+2)}
	$$
	and $T_n=T_{2m}\leq 0$.
	\medskip
	\\
	{\bf Case 2.} Now let $n$ be odd. In this case
	$$
			T_n:=\frac{g_0\,g_{n+2}}{(n!)^2(n+1)(n+2)}+
			\sum\limits_{k=0}^{\frac {n-1}2} \frac{g_{\frac {n+3} 2+k}\,g_{\frac {n+1}2-k}}
		{\left(\left(\frac  {n+3} 2+k\right)!\right)^2\left(\left(\frac  {n+1} 2-k\right)!\right)^2}(-n+4k^2+4k-1).
	$$
	As before, we denote
	$$
	t_k = \frac{g_{\frac {n+3} 2+k}\,g_{\frac {n+1}2-k}}{\left(\left(\frac  {n+3} 2+k\right)!\right)^2
		\left(\left(\frac  {n+1} 2-k\right)!\right)^2}(-n+4k^2+4k-1)
	$$
	and introduce the sequence
	$$
	s_0 = t_0 = -\frac{2g_{\frac {n+3} 2} \, g_{\frac {n+1}2}}{\left(\left(\frac  {n+3} 2\right)!\right)^2\left(\left(\frac  {n-1} 2\right)!\right)^2\left(\frac{n+1}2\right)},
	$$
	$$
	s_k = s_{k-1}\cdot \frac{\left(\frac{n+3}2-k\right)}{\left(\frac{n+3}2+k\right)}\cdot \frac{g_{\frac{n+3}2+k}\,
		g_{\frac{n+1}2-k}}{g_{\frac{n+1}2+k} \, g_{\frac{n+3}2-k}}+t_k.
	$$
	It can be shown by induction that
	$$
	s_k = -\frac{2(k+1)g_{\frac{n+3}2+k}\, g_{\frac{n+1}2-k}}{\left(\left(\frac  {n+3} 2+k\right)!\right)^2\left(\left(\frac  {n-1} 2-k\right)!\right)^2\left(\frac{n+1}2-k\right)}.
	$$
	Using the same considerations as for the even case, we get
	$$
	T_n\leq \frac{g_0\,g_{n+2}}{(n!)^2(n+1)(n+2)}+s_{\frac{n-1}2}.
	$$
	The property \eqref{eqns} gives
	$$
	s_{\frac{n-1}2}=-\frac{(n+1)g_{n+1}\,g_1}{((n+1)!)^2}\leq -\frac {g_0\,g_{n+2}}{(n!)^2(n+1)(n+2)}.
	$$
	Therefore, $T_n\leq 0$ for both cases. This completes the proof.
\end{proof}

\begin{lem}\label{lem2}
	If the Bergman weight  $w$ satisfies inequality \eqref{cond} for all $n\geq 1$, then
	\begin{equation}\label{ind}
		\frac{h_{2(n+1)}}{n+1}\leq h_{2n}\ln\left(\sum\limits_{k=0}^{\infty}\frac{h_{2k}}{(k!)^2}\right),\qquad n\geq 1.
	\end{equation}
\end{lem}

\begin{proof}
	Let $n=1$.
	From the H\"older inequality, we get the estimate	
	$$
	h_2^n = \left(\int\limits_0^1\rho^3w(\rho)\,d\rho\right)^n \leq \int\limits_0^1 \rho^{2n+1}w(\rho)\,d\rho\cdot\left(\int\limits_0^1\rho w(\rho)\,d\rho\right)^{n-1}= h_{2n}
	$$
	(recall that $h_0=1$). 
	It also follows from condition \eqref{cond} with $m=1$, that 
	\begin{equation}
		\label{h2h4}
		h_4 \le \frac{2h_2^2}{h_2+1}. 
	\end{equation} 
	Hence,
	\begin{multline*}
		2h_{2}\ln\left(\sum\limits_{k=0}^{\infty}\frac{h_{2k}}{(k!)^2}\right)-h_4\geq 2h_{2}\ln\left(\sum\limits_{k=0}^{\infty}\frac{h_{2}^k}{(k!)^2}\right)-\frac{2h_2^2}{h_2+1}\ge \\  \ge
		2h_2\ln I_0(2\sqrt{h_2})-2h_2^2\left(1-\frac{h_2}4\right). 
	\end{multline*}
	Here $I_0$ is the modified Bessel function. Recall that the modified Bessel functions $I_n$,
	defined as
	$$
	I_n(z) = \sum\limits_{m=0}^{\infty}\frac{(z/2)^{2m+n}}{m!(m+n)!}, 
	$$
	are connected with the classical Bessel functions 
	$J_n$ by the equality $I_n(z) =e^{-\frac{in\pi}2}J_n(ze^{\frac{i\pi}2})$.
	
	In what follows we will use the following three properties of the functions $I_n$
	which can be found in \cite[\S 2.12]{wat}:
	\begin{equation}\label{bess1}
		I_n'(x) = I_{n+1}(x)+\frac{n}{x}I_n(x),
	\end{equation}
	\begin{equation}\label{bess2}
		\frac{2n}x I_n(x) = I_{n-1}(x)-I_{n+1}(x),
	\end{equation}
	\begin{equation}\label{bess3}
		2 I_n'(x) = I_{n-1}(x)+I_{n+1}(x).
	\end{equation}
	
	We denote $2\sqrt{h_2}=t$ and consider the function
	$$
	u_1(t) = 16\ln I_0(t)-4t^2+\frac{t^4}4.
	$$
	Since $0<h_2 \le 1$, our goal is to show that $u_1(t)\geq 0$ for $0\le t \le 2$. 
	By the property \eqref{bess1} for $n=0$, we have
	$$
	u_1'(t) = \frac{16I_1(t)}{I_0(t)}-8t+t^3.
	$$
	Since $I_0(t) >0$ for $t\ge 0$, we need to consider
	$$
	16I_1(t)-8tI_0(t)+t^3I_0(t).
	$$
	Using the property \eqref{bess2} with $n=1$, we get the expression
	$$
	t^3I_0(t)-8tI_2(t)=:tu_2(t)
	$$
	We take the derivative of $u_2(t)$ and use the properties \eqref{bess1}--\eqref{bess3} to obtain that
	\begin{multline*}
		u_2'(t) = 2tI_0(t)+t^2I_1(t)-4(I_1(t)+I_3(t))\geq  2tI_0(t)-4I_1(t)-4I_3(t) = \\ = 2tI_2(t)-4I_3(t)=2tI_2(t)-\frac23tI_2(t)+\frac23tI_4(t)=\frac43tI_2(t)+\frac23tI_4(t).
	\end{multline*}
	Since $I_n(x) >0$, $x>0$, it follows that
	$$
	u_2(t)\geq u_2(0) = 0
	$$
	and so $u_1$ also increases. Then
	$$
	u_1(t) \geq u_1(0) = 0.
	$$
	
	Thus, we proved that
	$$
	\frac{h_4}2\leq h_2\ln\left(\sum\limits_{k=0}^{\infty}\frac{h_{2k}}{(k!)^2}\right).
	$$
	Let us suppose that the inequality \eqref{ind} is proved for $n-1$. Obviously, \eqref{cond} implies
	$$
	\frac{h_{2n}}{h_{2(n-1)}}\geq \frac n{n+1}\frac{h_{2(n+1)}}{h_{2n}},\qquad n\geq 1.
	$$
	Hence,
	$$
	\frac{h_{2(n+1)}}{n+1} \leq \frac{h_{2n}^2}{nh_{2(n-1)}}\leq\frac{h_{2n}}{h_{2(n-1)}}\cdot h_{2(n-1)}\ln\left(\sum\limits_{k=0}^{\infty}\frac{h_{2k}}{(k!)^2}\right),
	$$
	which is exactly the desired inequality.
\end{proof}	
\bigskip


\section{Proofs of Theorems~\ref{mainth} and \ref{counter}}\label{sec4}

\begin{proof}[Proof of Theorem~\ref{mainth}]
	We consider the function
	$$
	\phi(q) = \ln\left(\sum\limits_{n=0}^{\infty}\frac{q^n}{(n!)^2}h_{2n}\right) - q \ln\left(\sum\limits_{n=0}^{\infty}\frac{h_{2n}}{(n!)^2}\right).
	$$
	Its derivative is
	$$
	\phi'(q) = \frac{\sum\limits_{n=0}^{\infty}\frac{q^n}{(n!)^2(n+1)}h_{2(n+1)}}{\sum\limits_{n=0}^{\infty}\frac{q^n}{(n!)^2}h_{2n}} - \ln\left(\sum\limits_{n=0}^{\infty}\frac{h_{2n}}{(n!)^2}\right).
	$$
	We also take the second derivative
	$$
	\phi''(q) = \frac{\sum\limits_{n=0}^{\infty}\frac{q^n}{(n!)^2(n+1)(n+2)}h_{2(n+2)}\cdot \sum\limits_{n=0}^{\infty}\frac{q^n}{(n!)^2}h_{2n}-\left(\sum\limits_{n=0}^{\infty}\frac{q^n}{(n!)^2(n+1)}h_{2(n+1)}\right)^2}{\left(\sum\limits_{n=0}^{\infty}\frac{q^n}{(n!)^2}h_{2n}\right)^2}.
	$$
	Multiplying the series, we see that the first term in the numerator is
	$$
	\sum\limits_{n=0}^{\infty}\left(\sum\limits_{k=0}^{n}\frac{h_{2(n-k)}\cdot h_{2(k+2)}}{((n-k)!)^2(k!)^2(k+1)(k+2)}\right) q^n,
	$$
	and the second is equal to
	$$
	\sum\limits_{n=0}^{\infty}\left(\sum\limits_{k=0}^{n}\frac{h_{2(n-k+1)}\cdot h_{2(k+1)}}{((n-k)!)^2(k!)^2(n-k+1)(k+1)}\right) q^n.
	$$
	Since \eqref{cond} implies \eqref{simpCond} for $g_n=h_{2n}$, we can apply Lemma~\ref{lem1} and see that for all $n\geq 0$
	$$
	\sum\limits_{k=0}^{n}\frac{h_{2(n-k)}\cdot h_{2(k+2)}}{((n-k)!)^2(k!)^2(k+1)(k+2)}\leq \sum\limits_{k=0}^{n}\frac{h_{2(n-k+1)}\cdot h_{2(k+1)}}{((n-k)!)^2(k!)^2(n-k+1)(k+1)}.
	$$
	It follows that $\phi''(q)\leq 0$. Therefore, $\phi'(q)\leq \phi'(1)$ when $q\geq 1$. We denote
	$$
	\psi = \sum\limits_{n=0}^{\infty}\frac{h_{2(n+1)}}{(n!)^2(n+1)}-\sum\limits_{n=0}^{\infty}\frac{h_{2n}}{(n!)^2} \cdot \ln\left(\sum\limits_{n=0}^{\infty}\frac{h_{2n}}{(n!)^2}\right).
	$$
	From Lemma~\ref{lem2} we know that
	$$	
	\frac{h_{2(n+1)}}{n+1}\leq h_{2n}\ln\left(\sum\limits_{k=0}^{\infty}\frac{h_{2k}}{(k!)^2}\right),\qquad n\geq 1.
	$$
	We proved that all the terms of the series
	$$
	\psi = \sum\limits_{n=0}^{\infty}\frac1{(n!)^2}\left(\frac{h_{2(n+1)}}{n+1}- h_{2n}\ln\left(\sum\limits_{k=0}^{\infty}\frac{h_{2k}}{(k!)^2}\right)\right),
	$$
	are negative except the case $n=0$. To get rid of this exceptional term, we add it to the term with $n=1$.
	\begin{multline*}
		h_2-\ln\left(\sum\limits_{k=0}^{\infty}\frac{h_{2k}}{(k!)^2}\right)+\frac{h_4}2-h_2\ln\left(\sum\limits_{k=0}^{\infty}\frac{h_{2k}}{(k!)^2}\right) = \\=
		h_2+\frac{h_4}2-(1+h_2)\ln\left(\sum\limits_{k=0}^{\infty}\frac{h_{2k}}{(k!)^2}\right)\leq \\ \leq h_2+\frac{h_4}2-(1+h_2)\ln\left(1+h_2+\frac{h_4}4\right)=:y(h_2,h_4).
	\end{multline*}
	The derivative of this expression with respect to $h_4$
	$$
	\frac{\partial y}{\partial h_4} = \frac{2+2h_2+h_4}{8+8h_2+2h_4}>0,
	$$
	so $y$ increases with respect to $h_4$. Recall that we have  $h_4 \le \frac{2h_2^2}{1+h_2}$
	(see \eqref{h2h4}). Therefore, it is sufficient to substitute this value into $y$: 
	$$
	y\left(h_2,\frac{2h_2^2}{1+h_2}\right)=h_2 + \frac{h_2^2}{1+h_2} - (1 + h_2) \ln\left(1 + h_2 + \frac{h_2^2}{2+2h_2}\right).
	$$
	The function $v(h_2) = y\left(h_2,\frac{2h_2^2}{1+h_2}\right)/(1+h_2)$ has the same sign as the expression $y\left(h_2,\frac{2h_2^2}{1+h_2}\right)$. Calculations show that its derivative equals to
	$$
	v'(h_2) = -\frac{h_2^2 (2 + 3 h_2 + 3 h_2^2)}{(1 +h_2)^3 (2 + 4h_2 + 3 h_2^2)} \leq 0.
	$$
	Then
	$$
	y\left(h_2,\frac{2h_2^2}{1+h_2}\right)\leq y(0,0) = 0
	$$
	for all $h_2\in(0,1)$. Thus, the term with $n=0$ does not affect the sign and $\psi<0$.
	
	We proved that $\phi'(q)<0$, and so
	$$
	\phi(q)\leq \phi(1) = \ln\left(\sum\limits_{n=0}^{\infty}\frac{h_{2n}}{(n!)^2}\right) -  \ln\left(\sum\limits_{n=0}^{\infty}\frac{h_{2n}}{(n!)^2}\right)=0.
	$$
	Since logarithm is an increasing function, we finally obtain the desired inequality
	$$
	\sum\limits_{n=0}^{\infty}\frac{q^n}{(n!)^2}h_{2n}\leq \left(\sum\limits_{n=0}^{\infty}\frac{h_{2n}}{(n!)^2}\right)^q.
	$$
\end{proof}

\begin{proof}[Proof of Theorem~\ref{counter}]
	Consider the weight
	$$
	w^*(\rho) = \left\{\begin{aligned}
		\frac{3}{2\rho},\qquad 0\leq \rho\leq \frac12,\\
		\frac{1}{2\rho},\qquad \frac12<\rho\leq 1.
	\end{aligned}
	\right.
	$$
	Then $h_n=\frac{1 + 2^{-n}}{2 (1 + n)}$. One can note that $h_0=1$, so this weight is admissble. It is not
	continuous but we always can approximate it with a continuous or even smooth monotonically decreasing 
	weight with arbitrarily close moments $h_n$.
	
	However, one can check numerically that the derivative of the function
	$$
	\psi(q) = \sum\limits_{n=0}^{\infty}\frac{q^n}{(n!)^2}h_{2n}- \left(\sum\limits_{n=0}^{\infty}\frac{h_{2n}}{(n!)^2}\right)^q
	$$	
	at $q=1$ is positive ($\approx0.0048$). So, the function $\psi$ is equal to zero when $q=1$ and then increases. 
	It means that this value will be positive at least on some interval $(1,1+\varepsilon)$, $\varepsilon>0$,
	and \eqref{bern} does not hold.
	Also numerical calculations show that $\psi(2) \approx 0.0105 >0$.
	\end{proof}
	
%



\begin{rem}
	{\rm As we have shown at the end of Section 2, for the classical weights $w_{\alpha}$ inequality \eqref{cond}
		turns into equality. Let us show that all power weights $w(\rho) = (m+2)\rho^m$, $m\geq 0$ (not necessarily integer), 
		also satisfy \eqref{cond}. Indeed, we have $h_{2n} = \frac{2 + m}{2 + m + 2 n}$ and
		$$
			\frac{h_{2n}}{h_{2(n-1)}}-\left(\frac{h_{2(n+1)}}{(n+1)h_{2(n-1)}}+\frac n{n+1}\frac{h_{2(n+1)}}{h_{2n}}\right) 
			=  \frac{2 m}{(1 + n) (2 + m + 2 n) (4 + m + 2 n)} \ge 0.
	   $$}
\end{rem}

%

\section*{Funding}
The results of Sections~\ref{sec3} and \ref{sec4} (Theorem~\ref{mainth}) were obtained with the support of
Russian Science Foundation grant 22-71-10094. Other results of the paper were 
obtained with the support of Ministry of Science and
Higher Education of the Russian Federation, agreement No 075-15-2021-602.



\end{document}